\documentclass[10pt,a4paper]{article}
\usepackage{amsmath}
\usepackage{amssymb}
\usepackage{amsfonts}
\usepackage{amsthm}
\usepackage{relsize}

\newtheorem{thm}{Theorem}
\newtheorem{lem}[thm]{Lemma}
\newtheorem{prop}[thm]{Proposition}
\newtheorem{cor}[thm]{Corollary}

\newcommand{\Alt}{\mathrm{Alt}}

\newcommand{\Out}{\mathrm{Out}}
\newcommand{\Fin}{\mathrm{Fin}}
\newcommand{\Core}{\mathrm{Core}}

\newcommand{\rQ}{\mathrm{Q}}

\newcommand{\bC}{\mathbb{C}}
\newcommand{\bF}{\mathbb{F}}
\newcommand{\bN}{\mathbb{N}}

\newcommand{\bQ}{\mathbb{Q}}
\newcommand{\bR}{\mathbb{R}}
\newcommand{\bZ}{\mathbb{Z}}

\newcommand{\mcB}{\mathcal{B}}

\newcommand{\mcI}{\mathcal{I}}

\begin{document}

\title{Subgroups of finite index and the just infinite property}

\author{Colin Reid\\
School of Mathematical Sciences\\
Queen Mary, University of London\\
Mile End Road, London E1 4NS\\
c.reid@qmul.ac.uk}

\maketitle

\begin{abstract}A residually finite (profinite) group $G$ is just infinite if every non-trivial (closed) normal subgroup of $G$ is of finite index.  This paper considers the problem of determining whether a (closed) subgroup $H$ of a just infinite group is itself just infinite.  If $G$ is not virtually abelian, we give a description of the just infinite property for normal subgroups in terms of maximal subgroups.  In particular, we show that such a group $G$ is hereditarily just infinite if and only if all maximal subgroups of finite index are just infinite.  We also obtain results for certain families of virtually abelian groups, including all virtually abelian pro-$p$ groups and their discrete analogues.\end{abstract}


All groups under consideration will be topological groups.  We have the convention throughout that subgroups are required to be closed, that homomorphisms are required to be continuous, and that generation refers to topological generation.  When we wish to suppress topological considerations, the word `abstract' will be used, for instance `abstract subgroup'.\\

Say a group $G$ is \emph{residually finite} if the intersection of all normal subgroups of finite index is trivial.  Say $G$ is \emph{just infinite} if it is infinite and residually finite, and every non-trivial normal subgroup of $G$ is of finite index.  Say $G$ is \emph{hereditarily just infinite} if every finite index subgroup of $G$ is just infinite, including $G$ itself.  Our main aim in this paper will be to prove the following:

\begin{thm}\label{maxcor}Let $G$ be a just infinite group that is not virtually abelian.  Let $H$ be a non-trivial normal subgroup of $G$.  The following are equivalent:

(i) $H$ is just infinite;

(ii) Every maximal subgroup of $G$ containing $H$ is just infinite.\\

$(*)$ In particular, $G$ is hereditarily just infinite if and only if every maximal subgroup of finite index is just infinite.\end{thm}

The hypothesis that $G$ is not virtually abelian is essential, even for $(*)$.  However, the following result does include the virtually abelian case, as will be proved later.  Given a prime $p$, let $O^p(G)$ be the intersection of normal subgroups of $G$ of index a power of $p$.  The just infinite groups of quaternionic type are a family of virtually abelian pro-$2$ groups that first emerged in the classification of pro-$p$ groups of finite coclass (see \cite{Lee}); these will be defined later.

\begin{thm}\label{respthm}Let $G$ be a profinite or discrete group with no non-trivial finite normal subgroups, let $p$ be a prime, and let $H$ be a subgroup of finite index such that:

(i) $O^p(H)=1$;

(ii) every maximal subgroup of $H$ of finite index is just infinite;

(iii) $H$ is not of quaternionic type.\\

Then $G$ is hereditarily just infinite.\end{thm}

The following example shows that the word `normal' in the statement of Theorem \ref{maxcor} cannot be replaced with `finite index', even in the case that $G$ is a pro-$p$ group.

\paragraph{Example 1}Let $A$ be the group $V \rtimes C_p$, where $V$ is a vector space of dimension $p$ over $\bF_p$, and $C_p$ acts by permuting a basis of $V$.  There is a natural affine action of $A$ on $V$, extending the right regular action of $V$ on itself.  Let $K$ be any just infinite group that is not virtually abelian, and let $G$ be the permutation wreath product $K \wr_V A$ of $K$ by $A$ acting on $V$, where $A$ acts in the way described.  Note that any group of the form $K \wr_\Omega P$, where $K$ is just infinite and $P$ acts faithfully and transitively on the finite set $\Omega$, is necessarily just infinite.  In particular, $G$ is just infinite, with a subgroup $H$ of index $p^2$ of the form $K^{p^p} \rtimes W$, where $W$ is a subgroup of $V$ of index $p$ that is not normal in $A$.  Clearly $H$ is not just infinite, as $W$ does not act transitively on the copies of $K$; however, the unique maximal subgroup $M$ of $G$ containing $H$ is of the form $K \wr_V V$; since $V$ acts regularly on itself, this $M$ is just infinite.

\paragraph{Definition}The \emph{finite radical} $\Fin(G)$ of a residually finite group $G$ is the union of all finite normal subgroups of $G$.  Note that a residually finite topological group is necessarily a Hausdorff space, and so all finite abstract subgroups are closed.  However, the finite radical itself need not be closed.

\begin{lem}[Dicman \cite{Dic}]Let $G$ be an abstract group, and let $H$ be the subgroup generated by a set of elements $X$ that is a union of conjugacy classes.  Suppose that $X$ is finite and each element of $X$ is of finite order.  Then $H$ is finite.\end{lem}

\begin{cor}Let $G$ be a group and let $x \in G$.  Then $x \in \Fin(G)$ if and only if $x$ has finite order and $C_G(x)$ has finite index.\end{cor}

\begin{proof}Let $x$ be an element of finite order and centraliser of finite index, let $X$ be the set of conjugates of $x$, and let $H$ be the group generated by $X$.  Then $H$ is finite by Dicman's lemma, and by construction it is normal in $G$; thus $x \in \Fin(G)$.  The converse is clear.\end{proof}

\begin{lem}\label{nothji}Let $G$ be a residually finite group with $\Fin(G)=1$, and let $H$ be a finite index subgroup of $G$.  Then $\Fin(H)=1$.  If $H$ is just infinite then $G$ is just infinite, and if $H$ is hereditarily just infinite then $G$ is hereditarily just infinite.\end{lem}

\begin{proof}Any element $x$ of $\Fin(H)$ must have finite order and finitely many conjugates in $H$, and hence also in $G$.  This means $x \in \Fin(G)$, and so $x=1$.  We may now assume that $G$ is not hereditarily just infinite.  Then there is a subgroup $L$ of $G$ of finite index, with an infinite normal subgroup $K$ of infinite index.  This gives a non-trivial normal subgroup $K \cap H$ of $L \cap H$ of infinite index.  As $L \cap H$ has finite index in $H$, this means that $H$ is not hereditarily just infinite.  If $G$ is not just infinite, we can take $L=G$, and so $H$ cannot be just infinite.\end{proof}

\begin{cor}\label{jiup}Let $G$ be a just infinite group, and let $H$ be a just infinite subgroup of $G$ of finite index.  Then every subgroup of $G$ containing $H$ is just infinite.\end{cor}

We now focus on groups that are not virtually abelian.  Much of the existing general theory of just infinite groups that are not virtually abelian is due to J. S. Wilson, who gives an excellent account of his results on the subject in \cite{WilNH}.  In particular, we will be using Wilson's concept of basal subgroups.

\paragraph{Definition}Given a group $G$ and a subgroup $H$, write $H^G$ for the normal closure of $H$ in $G$. A non-trivial subgroup $B$ of $G$ is \emph{basal} if 
\[ B^G = B_1 \times \dots \times B_n,\]
for some $n \in \bN$, where $B_1, \dots, B_n$ are the conjugates of $B$ in $G$.  If $B$ is a basal subgroup of $G$, write $\Omega_B$ for the set of conjugates of $B$ in $G$, equipped with the conjugation action of $G$.  (Unless otherwise stated, we will always define $\Omega_B$ in terms of the group $G$.)

\begin{lem}\label{baslem}Let $G$ be a just infinite group that is not virtually abelian.  Let $K$ be a non-trivial subgroup of $G$ such that $K \lhd K^G$, and let $\{K_i \; i \in I\}$ be the set of conjugates of $K$ in $G$; given $J \subseteq I$, define $K_J := \bigcap \{K_j \;|\;j \in J\}$.  Then $I$ is finite, and there some $J \subseteq I$ such that $K_J$ is basal.\end{lem}

\begin{proof}Clearly $K^G$ has finite index in $G$, so $I$ is finite.  Now let $\mcI$ be the set of those $I' \subseteq I$ for which $K_{I'}$ is non-trivial, let $J$ be an element of $\mcI$ of largest size, and set $B=K_J$.  Suppose $B^g$ is a conjugate of $B$ distinct from $B$.  Then $B^g$ is of the form $K_{J'}$ where $|J'|=|J|$; by construction, this means $B \cap B^g = 1$, and also that $B$ and $B^g$ normalise each other, so in fact $[B,B^g]=1$.  Let $L=B^G$; now $C_L(B)$ contains all conjugates of $B$ apart from $B$ itself, so $L=BC_L(B)$.  Finally, note that $G$ has no non-trivial abelian normal subgroups, so $Z(L)=1$; hence $L=B \times C_L(B)$.  By symmetry, this means $L$ is a direct product of the conjugates of $B$, so $B$ is basal.\end{proof}

The following proposition will quickly lead to a proof of Theorem \ref{maxcor}. 

\begin{prop}\label{permji}Let $G$ be a just infinite group that is not virtually abelian, and let $H$ be a subgroup of $G$ of finite index.  Let $\mcB$ be the set of basal subgroups of $G$.  Then the following are equivalent:

(i) $H$ is not just infinite;

(ii) there is some $B \in \mcB$ such that $H$ acts intransitively on $\Omega_B$ and $\Core_G(H)$ acts trivially.\end{prop}

\begin{proof}Assume (i), and let $R$ be a non-trivial normal subgroup of $H$ of infinite index.  By Lemma \ref{nothji}, $R$ is infinite.  This means that $K=R \cap \Core_G(H)$ is an infinite normal subgroup of $H$ of infinite index such that $K \lhd K^G$.  By Lemma \ref{baslem}, there is a basal subgroup $B$ of $G$ which is an intersection of conjugates of $K$; by conjugating in $G$ if necessary, we may assume $B \leq K$.  Now $\Core_G(H)$ normalises $K$ and hence every conjugate of $K$, so $\mathrm{Core}_G(H)$ acts trivially on $\Omega_B$.  Furthermore, not all conjugates of $B$ are contained in $K$, as $K$ has infinite index, but the conjugates of $B$ contained in $K$ form a union of $H$-orbits on $\Omega_B$, which is non-empty as $B \leq R$.  So $H$ acts intransitively on $\Omega_B$ as required for (ii).\\

Assume (ii), and let $R=B^H$.  Then $R$ is an infinite subgroup of $B^G$ of infinite index, since $H$ acts intransitively on $\Omega_B$, so $R \cap H$ is an infinite subgroup of $H$ of infinite index; moreover, $R \cap H$ is normal in $H$.  Hence $H$ is not just infinite, which is (i).\end{proof}

\begin{proof}[Proof of Theorem \ref{maxcor}]Suppose $H$ is not just infinite.  Then by Proposition \ref{permji}, there is some non-normal basal subgroup $B$ of $G$ such that $H$ acts trivially on $\Omega_B$.  It follows that there is a maximal subgroup $M$ of $G$ containing $H$ which acts intransitively on $\Omega_B$, and hence $M$ is also not just infinite.  The converse follows immediately from Corollary \ref{jiup}.  Since every finite index subgroup contains a finite index normal subgroup, $(*)$ also follows.\end{proof}

We now discuss some consequences of Theorem \ref{maxcor}.\\

Let $\Phi_f(G)$ denote the intersection of all maximal subgroups of $G$ of finite index.  Note that if $G$ is a profinite group then $\Phi_f(G)$ is the intersection $\Phi(G)$ of all maximal closed subgroups, and if $G$ is a pro-$p$ group then $G/\Phi(G)$ is the largest elementary abelian image of $G$, which is finite in this case if and only if $G$ is finitely generated.

\begin{cor}Let $G$ be a residually finite group such that $\Phi_f(G)$ has finite index in $G$, and such that $G$ is not virtually abelian.  Then $G$ is hereditarily just infinite if and only if $\Phi_f(G)$ is just infinite and $\Fin(G)$ is trivial.\end{cor}

\begin{proof}If $\Phi_f(G)$ is just infinite and $\Fin(G)$ is trivial, then every maximal finite index subgroup is just infinite by Corollary \ref{jiup}, and so $G$ is hereditarily just infinite by the theorem.  The converse is immediate.\end{proof}

In general, a just infinite group $G$ may have infinitely many or finitely many maximal subgroups of finite index, with $\Phi_f(G)$ respectively trivial or of finite index.  In the profinite case, this is determined entirely by whether or not the group in question is virtually pronilpotent, where a profinite group is said to be pronilpotent if it is the inverse limit of finite nilpotent groups.

\begin{lem}Let $G$ be a finite group, acting faithfully and primitively on a finite set $\Omega$, and let $H$ be a normal subgroup of $G$.  Then $\Phi(H)=1$.\end{lem}

\begin{proof}Note first that for any normal subgroup $N$ of $G$, the $N$-orbits define a $G$-invariant equivalence relation on $\Omega$, so either $N=1$ or $N$ acts transitively; in particular, since $\Phi(H)$ is normal in $G$ it suffices to show that $\Phi(H)$ acts intransitively.  We may assume $H$ acts transitively, and so $\Omega$ can be regarded as the set of right cosets of some subgroup $K$ of $H$, acted on by right multiplication.  Let $M$ be a maximal subgroup of $H$ containing $K$.  Then $\{Km \; |m \; \in M\}$ is an $M$-orbit, so $M$ acts intransitively.  Hence $\Phi(H) \leq M$ also acts intransitively.\end{proof}

\begin{prop}Let $G$ be a just infinite profinite group.  The following are equivalent:

(i) $G$ is virtually pronilpotent;

(ii) $G$ is virtually pro-$p$, for some prime $p$;

(iii) $G$ has finitely many maximal open subgroups.\end{prop}

\begin{proof}Assume (i), and let $N$ be the largest pronilpotent normal subgroup of $G$.  Then $N$ is non-trivial and every Sylow subgroup of $N$ is characteristic in $G$.  Let $S$ be a Sylow pro-$p$ subgroup of $N$ and $T$ a Sylow pro-$q$ subgroup of $N$, for distinct primes $p$ and $q$.  Then at least one of $S$ and $T$ has infinite index in $G$, as their intersection is trivial.  It follows that $N$ is a pro-$p$ group for some $p$, and hence $G$ is virtually pro-$p$.  This is (ii).\\

Assume (ii), and let $N$ be the core of a maximal subgroup of $G$ of finite index.  Then $G/N$ admits a faithful primitive permutation action on some set $\Omega$.  This means that $\Phi(O_p(G/N))=1$, by the above lemma.  Hence $O_p(G/N)$ is elementary abelian, and so the same is true for its subgroup $O_p(G)N/N$; hence $\Phi(O_p(G)) \leq N$.  As this holds for all cores of maximal finite index subgroups, we have $\Phi(O_p(G)) \leq \Phi(G)$.  Furthermore, $\Phi(O_p(G))$ is normal in $G$, and so has finite index, since otherwise $O_p(G)$ would be elementary abelian, contradicting the fact that $\Fin(O_p(G))=1$.  Hence $\Phi(G)$ is of finite index.  Since only finitely many subgroups can contain a given subgroup of finite index, this implies (iii).\\

Assume (iii).  Then $\Phi(G)$ is an open pronilpotent normal subgroup of $G$.  (See \cite{Wil}.)  Hence (i) holds.\end{proof}

\begin{cor}Let $G$ be a just infinite profinite group that is virtually pronilpotent but not virtually abelian.  Then $G$ is hereditarily just infinite if and only if $\Phi(G)$ is just infinite.\end{cor}

\paragraph{}We now consider virtually abelian groups, with regard to the question of whether statement $(*)$ of Theorem \ref{maxcor} generalises to such groups.  For simplicity, we will only consider groups that are either discrete or profinite.\\

The following is well-known; see \cite{McC} for a more detailed discussion of the discrete case.

\begin{prop}\label{vastr}Let $G$ be an infinite virtually abelian group that is either discrete or profinite.  Then $G$ is just infinite if and only if the following conditions are satisfied:

(i) there is a self-centralising normal subgroup $A$ of $G$ of finite index, such that $A \cong O^d$ for some integer $d$, where $O=\bZ$ in the discrete case and $O=\bZ_p$ for some $p$ in the profinite case;

(ii) $G/A$ acts as on $A$ as matrices over $O$;

(iii) $G/A$ is irreducible as a matrix group over $F$, where $F$ is the field of fractions of $O$.\end{prop}

A key observation in establishing Theorem \ref{maxcor} was that every intransitive normal subgroup of a finite permutation group is contained in an intransitive maximal subgroup.  Similarly, an imprimitive finite linear group has a maximal subgroup that is reducible, which leads to the following:

\begin{cor}\label{imprim}In the situation of Proposition \ref{vastr}, suppose that $G/A$ is imprimitive as a matrix group over $F$.  Then $G$ has a maximal subgroup of finite index which is not just infinite.\end{cor}

However, there are primitive finite linear groups, all of whose maximal subgroups are irreducible, and so statement $(*)$ does not generalise completely to the virtually abelian case.  Let $A \cdotp B$ denote any group $G$ such that $A \unlhd G$ and $G/A \cong B$.

\paragraph{Example 2}(The author thanks Charles Leedham-Green for pointing out this example.)  Let $\rQ_{2^n}$ denote the generalised quaternion group of order $2^n$.  In Examples 10.1.18 of \cite{Lee}, the authors give a $4$-dimensional representation of $\rQ_{16}$ over $\bQ_2$, giving rise to just infinite pro-$2$ groups $G$ of the form $A \cdotp \rQ_{16}$ where $A \cong \bZ^4_2$ and $G/A$ acts faithfully on $A$.  Say such a group $G$ is of \emph{quaternionic type}.  Both $\rQ_{16}$ and all its maximal subgroups act irreducibly over $\bQ_2$, and so every maximal open subgroup of $G$ is just infinite.  However, $\Phi_f(G)$ is not just infinite; indeed, we note that an open subgroup $K$ of $G$ is just infinite if and only if $|G/A:KA/A| \leq 2$.\\

It turns out that the above example describes the only way in which $(*)$ can fail for pro-$p$ groups, and there are no exceptions to $(*)$ among discrete groups for which $O^p(G)=1$; this leads to Theorem \ref{respthm}.  Case (i) of the following is Theorem 10.1.25 of \cite{Lee}, and case (ii) can easily be deduced from results in section 10.1 of \cite{Lee}:

\begin{thm}[Leedham-Green, McKay \cite{Lee}]\label{leethm}Let $G$ be a finite $p$-group with a faithful primitive representation over a field $F$.\\

(i) If $F=\bQ_p$, then either $G$ is $C_p$ acting in dimension $p-1$, or $p=2$ and $G$ is $\rQ_{16}$ acting in dimension $4$.\\

(ii) If $F$ is a subfield of $\bR$ which does not contain $\sqrt{2}$, then $G$ is $C_p$ acting in dimension $p-1$.\end{thm}

\begin{proof}[Proof of Theorem \ref{respthm}]By Lemma \ref{nothji}, it suffices to prove that $H$ is hereditarily just infinite.  By Theorem \ref{maxcor}, we may assume $H$ is virtually abelian, and hence of the form described in Proposition \ref{vastr}.  Given our hypotheses, the result now follows immediately from Theorem \ref{leethm} together with Corollary \ref{imprim}.\end{proof}

Once we allow multiple primes to appear, there is much more potential for primitive finite linear groups, and hence for just infinite discrete or profinite groups which do not obey $(*)$.  We give one example here, but it seems likely that there are many others, with no straightforward way of classifying them.
 
\paragraph{Example 3}(The author thanks John Bray for pointing out this example; more details can be found in \cite{Gri}.)  Let $E$ be the extraspecial group of order $2^7$ that is a central product of dihedral groups of order $8$.  Then there is a group $L$ containing $E$ as a normal subgroup with $C_L(E)=Z(E)$, such that $L/E \cong \Out(E) \cong \Alt(8)$, and the smallest supplement to $E$ in $L$ is the double cover $2 \cdotp \Alt(8)$ of $\Alt(8)$.  Furthermore, $L$ has a faithful irreducible $8$-dimensional representation over $\bC$, which can in fact be realised over $\bZ$.  Let $G$ be the corresponding semidirect product $O^8 \rtimes L$, where $O$ is $\bZ$ or $\bZ_p$ for some $p$; then all maximal subgroups of $G$ of finite index contain a group of the form $O^8 \rtimes E$ or $O^8 \rtimes (2 \cdotp \Alt(8))$.  For both $E$ (see for instance \cite{Doe}) and $2 \cdotp \Alt(8)$ (see \cite{ATL}), the smallest faithful representation in characteristic $0$ has dimension $8$.  It follows that $G$ is a just infinite group, which can be chosen to be either discrete or virtually pro-$p$ with no restrictions on the prime $p$, and every maximal subgroup of $G$ of finite index is just infinite in all cases.  However, $G$ is evidently not hereditarily just infinite.

\section*{Acknowledgments}This paper is based on results obtained by the author while under the supervision of Robert Wilson at Queen Mary, University of London (QMUL).  The author would also like to thank Charles Leedham-Green for introducing the author to the study of just infinite groups, for his continuing advice and guidance in general, and for his specific advice and corrections concerning this paper; Yiftach Barnea for informing the author about recent developments in just infinite profinite groups; and Peter Cameron for his advice on permutation groups.  The author acknowledges financial support provided by EPSRC and QMUL for the duration of his doctoral studies.

\end{document}